\documentclass[12pt]{amsart}%
\usepackage{amscd,amsmath,latexsym,amsthm,amsfonts,amssymb,graphicx,color,geometry}
\usepackage{enumerate}
\usepackage[pdftex,plainpages=false,pdfpagelabels,
colorlinks=true,citecolor=blue,linkcolor=black,urlcolor=black,
filecolor=black,bookmarksopen=true,unicode=true]{hyperref}
\usepackage[norefs,nocites]{refcheck}
\usepackage{amsmath}
\usepackage{amsfonts}
\usepackage{amssymb}
\usepackage[T1]{fontenc}
\usepackage{graphicx}
\usepackage{xcolor}
\usepackage{colortbl}
\usepackage{float}
\usepackage{framed}
\usepackage{pgf,tikz,pgfplots}
\usepackage{tkz-base}
\usepackage{multicol}
\usepackage{tabularx}
\usepackage{tkz-fct}%
\providecommand{\U}[1]{\protect\rule{.1in}{.1in}}
\tikzset{>=Triangle}
\newtheorem{theorem}{Theorem}[section]

\newtheorem{corollary}[theorem]{Corollary}
\newtheorem{example}[theorem]{Example}

\newtheorem{definition}[theorem]{Definition}
\numberwithin{equation}{section}
\pgfplotsset{compat=1.17}
\begin{document}
\title[Pointwise linear separation property and infinite pointwise dense lineability]{Pointwise linear separation property and infinite pointwise dense lineability}
\author[A. Raposo Jr.]{Anselmo Raposo Jr.}
\address{Departamento de Matem\'{a}tica \\
Universidade Federal do Maranh\~{a}o \\
65085-580 - S\~{a}o Lu\'{\i}s, Brazil.}
\email{anselmo.junior@ufma.br}
\author[G. Ribeiro]{Geivison Ribeiro}
\address{Departamento de Matem\'{a}tica \\
Universidade Federal da Para\'{\i}ba \\
58.051-900 - Jo\~{a}o Pessoa, Brazil.}
\email{geivison.ribeiro@academico.ufpb.br}

\begin{abstract}
In this note we generalize a criterion within the concept of infinite dense
lineability due to Calder\'{o}n-Moreno, Gerlach-Mena and Prado-Bassas. We also
introduce and explore some "local" notions of lineability.

\end{abstract}
\maketitle

\section{Introduction}

In 1872, K. Weierstrass shows that, if $0<a<1$, $b$ is an odd positive integer
and $ab>1+3\pi/2$, then the function $f_{a,b}\colon\left[  0,1\right]
\rightarrow\mathbb{R}$ defined by%
\[
f_{a,b}\left(  x\right)  =%
%TCIMACRO{\tsum \limits_{n=0}^{\infty}}%
%BeginExpansion
{\textstyle\sum\limits_{n=0}^{\infty}}
%EndExpansion
a^{n}\cos\left(  b^{n}\pi x\right)
\]
is continuous over the interval $\left[  0,1\right]  $, but it is non
differentiable in any of its points. As far as we know, this is the first
published example of a class of functions with such properties, in which case
the $f_{a,b}$ functions are then nicknamed Weierstrass monster functions.
Later, in 1966 (see \cite{Gurariy}), V. Gurariy constructs an
infinite-dimensional vector subspace consisting, up to the zero vector, of
continuous functions that are not differentiable anywhere. In 1995 L.
Rodr\'{\i}guez-Piazza proves in \cite{Piazza} that every separable Banach
space is isometric to a space of continuous functions that are not
differentiable anywhere and, in 2005, R. Aron, V. Gurariy and J.
Seoane-Sep\'{u}lveda (see \cite{AGS}) investigate similar problems in other
contexts, initiating the branch of research that we now know as
\textit{lineability}. For more details, we recommend reading \cite{AGPS}.

Roughly speaking, lineability is the branch of mathematics that is intended to
look for large linear structures in exotic subsets of vector spaces. More
precisely, if $V$ is a vector space, $M$ is a nonempty subset of $V$, and
$\alpha$ is a cardinal number, we say that $M$ is $\alpha$\textit{-lineable}
if there exists a subspace $W$ of $V$ such that%
\[
\dim\left(  W\right)  =\alpha\text{ \ \ \ \ and \ \ \ \ }W\subset
M\cup\left\{  0\right\}  \text{.}%
\]
When $V$ is topological vector space, we say that $M$ is $\alpha
$\textit{-spaceable} ($\alpha$\textit{-dense lineable} or \textit{densely}
$\alpha$\textit{-lineable}) if there is a subspace $W$ of $V$ such that%
\[
W\text{ is closed (dense), \ \ \ \ }\dim\left(  W\right)  =\alpha\text{
\ \ \ \ and \ \ \ \ }W\subset M\cup\left\{  0\right\}  \text{.}%
\]

With the development of the theory, it was verified that positive results of
lineability are quite common and, on the other hand, techniques and general
criteria do not appear with the same frequency. In this perspective, some more
restrictive notions of lineability arise, giving the problems interesting
geometric contours and adding information of a qualitative nature.

Recently, inspired by the notions of lineability presented by V. F\'{a}varo,
D. Pellegrino and D. Tom\'{a}z in \cite{FPT}, D. Pellegrino and A. Raposo Jr.
introduce the following pointwise (or directional) notion of lineability: a
nonempty subset $M$ of a vector space $V$ is said to be \textit{pointwise
}$\alpha$\textit{-lineable} if, for every $x\in M$, there exists a subspace
$W_{x}$ of $V$ such that%
\[
x\in W_{x}\text{, \ \ \ \ }\dim\left(  W_{x}\right)  =\alpha\text{ \ \ \ \ and
\ \ \ \ }W_{x}\subset M\cup\left\{  0\right\}  \text{.}%
\]
When $V$ is topological vector space and, for each $x\in M$, the subspace
$W_{x}$ can be chosen closed (dense), we say that $M$ is \textit{pointwise
}$\alpha$\textit{-spaceable} ($\alpha$\textit{-dense lineable }or\textit{
densely }$\alpha$\textit{-lineable}).

\section{A new approach: Linear separation theorems}

Since pointwise $\alpha$-lineability in a subset $M$ of a vector space $V$
concerns the existence of an $\alpha$-dimensional vector subspace $W_{x}$ in
$M\cup\left\{  0\right\}  $ for each vector $x\in M\cup\left\{  0\right\}  $,
we can consider the family $F_{\alpha}$ of all $\alpha$-dimensional vector
subspaces of $V$ in $M\cup\left\{  0\right\}  $. Hence, we have the following
natural question: we can ask whether for each sequence linearly independent
$\left(  x_{n}\right)  _{n=1}^{\infty}$ of elements of $V$ in $M$, there
exists a family $\mathcal{W}_{\alpha}=\left\{  W_{n}:n\in\mathbb{N}\right\}  $
of vector subspaces of $V$ such that:

\begin{enumerate}
\item[$\left(  i\right)  $] for each $n\in\mathbb{N}$, $\dim\left(
W_{n}\right)  =\alpha$,

\item[$\left(  ii\right)  $] for each $n\in\mathbb{N}$, $x_{n}\in W_{n}\subset
M\cup\left\{  0\right\}  $ and

\item[$\left(  iii\right)  $] $W_{m}\cap W_{n}=\left\{  0\right\}  $ whenever
$m$ and $n$ are distinct positive integers.
\end{enumerate}

In the same way as "pointwise lineability", this idea can be described as a
local version of the $\alpha$-lineability. Inspired by the separation property
to Hausdorff spaces we also introduce the notion of set with the pointwise
linear separation property and in particular, we provide sufficient and
necessary conditions for a set in a topological vector space to satisfy such a property.

\bigskip

\begin{definition}
Let $\lambda,\alpha$ be two cardinal numbers with $0<\lambda\leq\alpha$, $M$
be a nonempty subset of a topological vector space $V\neq\left\{  0\right\}  $
and $x=\left(  x_{i}\right)  _{i\in\lambda}$ be a set linearly independent of
elements of $V$ in $M$ with cardinality $\lambda$. We say that $M$ is $\left[
\left(  x_{i}\right)  _{i\in\lambda},\alpha\right]  $-lineable if it is
$\alpha$-lineable, and there exists a family $\mathcal{W}_{\alpha}=\left\{
W_{i}\right\}  _{i\in\lambda}$of vector subspaces of $V$ such that:

\begin{enumerate}
\item[$\left(  i\right)  $] for each $i\in\mathbb{\lambda}$, $\dim\left(
W_{i}\right)  =\alpha$,

\item[$\left(  ii\right)  $] for each $i\in\mathbb{\lambda}$, $x_{i}\in
W_{i}\subset M\cup\left\{  0\right\}  $,

\item[$\left(  iii\right)  $] $W_{i}\cap W_{j}=\left\{  0\right\}  $ whenever
$i$ and $j$ are distinct.
\end{enumerate}
\end{definition}

In addition, if $M$ is $\left[  \left(  x_{i}\right)  _{i\in\lambda}%
,\alpha\right]  $-lineable for each set linearly independent $\left(
x_{i}\right)  _{i\in\lambda}$ in $V$, we say that $M$ has $\left(
\lambda,\alpha\right)  $ pointwise linear separation property\textbf{
}$(\left(  \lambda,\alpha\right)  $-\textbf{P.L.S.P }in short$)$. When $V$ is
a topological vector space and each subspace $W_{i}$ in $\mathcal{W}_{\alpha}$
can be chosen dense, we say that $M$ is $\left[  \left(  x_{i}\right)
_{i\in\lambda},\alpha\right]  $-dense lineable, and more generally, that $M$
has $\left(  \lambda,\alpha\right)  $-dense pointwise linear separation
property\textbf{ }$(\left(  \lambda,\alpha\right)  $-\textbf{D}%
.\textbf{P.L.S.P }in short$)$ if it is $\left[  \left(  x_{i}\right)
_{i\in\lambda},\alpha\right]  $-dense lineable for each linearly independent
set $\left(  x_{i}\right)  _{i\in\lambda}$ of elements of $V$ in $M$.

\bigskip

Unless otherwise specified, throughout this paper $V$ will represent a given
Hausdorff topological vector space. Furthermore we will say that a subspace
$Z$ is \emph{transversal} to another subspace $W$ in $V$ whenever $Z\cap
W=\{0\}$. The letters $\alpha,\beta,\lambda$ will always represent cardinal
numbers, $\operatorname*{card}\left(  M\right)  $ will denotes the cardinality
of the set $M$, $\aleph_{0}:=\operatorname*{card}\left(  \mathbb{N}\right)  $
and $\mathfrak{c}:=\operatorname*{card}\left(  \mathbb{R}\right)  $ and
moreover, if $W$ is a subspace of $V$, the \textit{codimension} of $W$,
symbolically denoted by $\operatorname{codim}W$, will indicate the dimension
of the quotient space $V/W$.

\bigskip

We start by showing that pointwise $\alpha$-lineability does not imply having
$\left(  \alpha,\alpha\right)  $-\textbf{P.L.S.P}.

\begin{example}
Consider the following subset of $\ell_{1}$:%
\[
M:=\operatorname*{span}\left\{  e_{1},e_{2}\right\}  \cup\left\{  \left(
x_{n}\right)  _{n=1}^{\infty}\in\ell_{1}:x_{1}=x_{2}=0\right\}  \text{.}%
\]
The set $M$ is pointwise $2$-lineable, but no satisfy $\left(  2,2\right)
$-\textbf{P.L.S.P}.
\end{example}

Furthermore, if we consider the usual sequence $\left(  e_{n}\right)
_{n=1}^{\infty}$ in $\ell_{\infty}$, where $e_{n}=$ $(0,0,\ldots
0,1,0,0,\ldots)$ (with the $1$ at the nth place) we can conclude that
$M:=\operatorname*{span}\left\{  e_{n}:n\in\mathbb{N}\right\}  \setminus
\left\{  0\right\}  $ is $\left[  \left(  x_{i}\right)  _{i\in\aleph_{0}%
},\aleph_{0}\right]  $-lineable for some set $\left(  x_{i}\right)
_{i\in\aleph_{0}}$.\textbf{ }In fact, let%
\[
\mathbb{N=}\bigcup\limits_{k=1}\mathbb{N}_{k}%
\]
with $\mathbb{N}_{i}\cap\mathbb{N}_{j}=\varnothing$ whenever $i\neq j$ and
$\operatorname{card}\left(  \mathbb{N}_{k}\right)  =\aleph_{0}$ for all $k$.
Denote%
\[
\mathbb{N}_{k}=\left\{  n_{j}^{\left(  k\right)  }:j\in\mathbb{N}\right\}
\]
with $n_{i}^{\left(  k\right)  }<n_{j}^{\left(  k\right)  }$whenever $i<j$ and
define%
\[
E_{k}:=\operatorname*{span}\left\{  \left\{  e_{n_{k}^{\left(  1\right)  }%
}\right\}  \cup\left\{  e_{n}:n\in\mathbb{N}_{k+1}\right\}  \right\}  \text{
for each }k\text{.}%
\]
It is plain that $M:=\operatorname*{span}\left\{  e_{n}:n\in\mathbb{N}%
\right\}  \setminus\left\{  0\right\}  $ is $\left[  \left(  x_{i}\right)
_{i\in\aleph_{0}},\aleph_{0}\right]  $-lineable $($if we take $x_{i}:=e_{i}$
for each $i\in\mathbb{N}_{1})$. However, $M$ does not have $\left(  \aleph
_{0},\aleph_{0}\right)  $-\textbf{P.L.S.P}.

\bigskip

The result below is due to F\'{a}varo et al. in \cite{Raposo} and is inspired
by \cite{Leo}. This will be essential to prove Theorem
\ref{Primeiro resultado}.

\begin{theorem}
\label{Pams} \cite[Theorem 4.2]{Raposo} Let $V\neq\left\{  0\right\}  $ and
$W\subset V$ be a linear subspace such that $w(V)\leq\dim\left(  V/W\right)
$. Then $V\setminus W$ is $\left(  \alpha,\beta\right)  $-dense lineable for
each $\alpha<\dim\left(  V/W\right)  $ and%
\[
\max\left\{  \alpha,w\left(  V\right)  \right\}  \leq\beta\leq\dim\left(
V/W\right)  .
\]

\end{theorem}

\begin{theorem}
\label{Primeiro resultado}Let $V\neq\left\{  0\right\}  $ and let
$\alpha>w\left(  V\right)  \geq\aleph_{0}$ be a cardinal number. Let $M$ be a
nonempty subset of $V$. If $M$ is pointwise $\alpha$-dense lineable then $M$
has $\left(  2,w\left(  V\right)  \right)  $-\textbf{D}.\textbf{P.L.S.P}.

\begin{proof}
Let $x,y\in M$ be two linearly independent vectors of $V$. Since $M$ is
pointwise $\alpha$-dense lineable, there is an $\alpha$-dimensional vector
subspace $W_{x}$ dense in $V$ such that%
\[
x\in W_{x}\text{ \ \ \ and \ \ \ }W_{x}\subset M\cup\left\{  0\right\}
\text{.}%
\]
If $y\in W_{x}$ consider a Hamel basis $\left\{  x_{a}:a\in I\right\}  $ to
$W_{x}$ containing $x$ and $y$. Let $I=I_{1}\cup I_{2}$ be a partition of $I$
into two subsets of cardinality $\alpha$ with $x=x_{a_{1}}$ for some $a_{1}\in
I_{1}$ and $y=x_{a_{2}}$ for some $a_{2}$ in $I_{2}$. Let%
\begin{equation}
W_{x,1}:=\operatorname*{span}\left\{  x_{a}:a\in I_{1}\right\}  \text{.}%
\label{U}%
\end{equation}
Thus, we get $\mathbb{K}y\cap W_{x,1}=\left\{  0\right\}  $ and since $\dim
W_{x}/W_{x,1}=\alpha>w\left(  V\right)  \geq w\left(  W_{x}\right)  $, we can
invoke Theorem \ref{Pams} to obtain a $w\left(  V\right)  $-dimensional vector
subspace $W_{y}$ dense in $W_{x}$ $($containing $\mathbb{K}y)$ such that
$W_{y}$ is transversal to $W_{x,1}$. That is,
\begin{equation}
W_{y}\cap W_{x,1}=\left\{  0\right\}  \text{.}\label{A}%
\end{equation}
The subspace $W_{x,1}$ defined in $(\ref{U})$ is not necessarily dense in
$W_{x}$. However, since $\dim W_{x}/W_{y}=\alpha>w\left(  V\right)  \geq
w\left(  W_{x}\right)  $ and $\mathbb{K}x\cap W_{y}=\left\{  0\right\}  $ we
can invoke Theorem \ref{Pams} again to obtain a $w\left(  V\right)
$-dimensional vector subspace $W_{x,2}$ dense in $W_{x}$ $($containing
$\mathbb{K}x)$ such that $W_{x,2}$ is transversal to $W_{y}$. By the fact that
$\left(  W_{x,2}\cup W_{y}\right)  \subset W_{x}\subset M\cup\left\{
0\right\}  $ and $W_{x}$ is dense in $V$, we conclude that the vector
subspaces $W_{x,2}$ and $W_{y}$ are both dense in $V$. Hence, the proof is
complete for the case where $y\in W_{x}$. Now, we will assume that%
\begin{equation}
y\notin W_{x}\text{.}\label{AA}%
\end{equation}
The fact that $M$ is pointwise $\alpha$-dense lineable ensures that there is
an $\alpha$-dimensional vector subspace $Z_{y}$ dense in $V$ such that%
\[
y\in Z_{y}\text{ \ \ \ and \ \ \ }Z_{y}\subset M\cup\left\{  0\right\}
\text{.}%
\]
Let $v\in W_{x}$ be such that $v\notin\mathbb{K}x$. Since $\dim W_{x}%
/\mathbb{K}v=\alpha>w\left(  V\right)  $ we can use Theorem \ref{Pams} again
to obtain a vector subspace $\mathcal{D}_{x}$ dense in $W_{x}$ such that%
\[
x\in\mathcal{D}_{x}\text{ \ \ \ and \ \ \ }\dim\mathcal{D}_{x}=w\left(
V\right)  \text{.}%
\]
If $\mathcal{D}_{x}\cap Z_{y}=\left\{  0\right\}  $, let $\mathcal{D}_{y}$ be
a vector subspace dense in $Z_{y}$ such that%
\[
y\in\mathcal{D}_{y}\text{ \ \ \ and \ \ \ }\dim\mathcal{D}_{y}=w\left(
V\right)  \text{.}%
\]
Hence%
\[
\mathcal{D}_{x}\cap\mathcal{D}_{y}=\left\{  0\right\}  \text{,}%
\]
and the result is done. Otherwise, if $\mathcal{D}_{x}\cap Z_{y}\neq\left\{
0\right\}  $, let $\mathcal{N}:=\mathcal{D}_{x}\cap Z_{y}$. Since
$y\notin\mathcal{N}$ $\left(  y\notin\mathcal{D}_{x}\right)  $ and
$\dim\mathcal{N}\leq\dim\mathcal{D}_{x}=w\left(  V\right)  <\alpha=\dim Z_{y}%
$, we can infer $($again by Theorem \ref{Pams}$)$ that there is a vector
subspace $\mathcal{D}_{1,y}$ dense in $Z_{y}$ such that%
\[
y\in\mathcal{D}_{1,y}\text{, \ \ }\dim\mathcal{D}_{1,y}=w\left(  V\right)
\text{ \ \ and \ \ }\mathcal{D}_{1,y}\cap\mathcal{N}=\left\{  0\right\}
\text{.}%
\]
In particular, we get%
\[
\mathcal{D}_{x}\cap\mathcal{D}_{1,y}=\left\{  0\right\}  \text{.}%
\]
Thus, the proof is complete.
\end{proof}
\end{theorem}

\bigskip

The result above can also be stated as follows:

\begin{theorem}
Let $V\neq\left\{  0\right\}  $ and let $\alpha_{2}>\alpha_{1}\geq w\left(
V\right)  $ be cardinal numbers and $M$ be a nonempty subset of $V$. If $M$ is
pointwise $\alpha_{2}$-dense lineable then $M$ has $\left(  2,\alpha
_{1}\right)  $-\textbf{D}.\textbf{P.L.S.P}.
\end{theorem}

The next result in this section characterizes sets with $\left(
2,\alpha\right)  $-\textbf{P.L.S.P} in the context of topological vector spaces.

\begin{theorem}
\label{First corollary in DPSP}Let $V\neq\left\{  0\right\}  $ be a vector
space and let $\alpha$ be an infinite cardinal number. Let $M$ be a nonempty
subset of $V$. Then $M$ is pointwise $\alpha$-lineable if and only if it has
$\left(  2,\alpha\right)  $-\textbf{P.L.S.P}.
\end{theorem}

\begin{proof}
Let $x,y\in M$ be two linearly independent vectors of $V$. Since $M$ is
pointwise $\alpha$-lineable, there is an $\alpha$-dimensional vector subspace
$W$ in $V$ such that%
\[
x\in W\text{ \ \ \ and \ \ \ }W\subset M\cup\left\{  0\right\}  \text{.}%
\]
If $y\in W$ consider a Hamel basis $\left\{  x_{a}:a\in I\right\}  $ to $W$
containing $x$ and $y$ and let $I=I_{1}\cup I_{2}$ be a partition of $I$ into
two subsets of cardinality $\alpha$ with $x=x_{a_{1}}$ for some $a_{1}\in
I_{1}$ and $y=x_{a_{2}}$ for some $a_{2}$ in $I_{2}$. The vector subspaces
$\mathcal{V}_{x}:=\operatorname*{span}\left\{  x_{a}:a\in I_{1}\right\}  $ and
$\mathcal{V}_{y}:=\operatorname*{span}\left\{  x_{a}:a\in I_{2}\right\}  $ are
such that%
\[
\mathcal{V}_{x}\cap\mathcal{V}_{y}=\left\{  0\right\}  \text{ \ \ and
\ \ }\dim\mathcal{V}_{x}=\dim\mathcal{V}_{y}=\alpha\text{.}%
\]
Furthermore,%
\[
\mathcal{V}_{x}\cup\mathcal{V}_{y}\subset W\subset M\cup\left\{  0\right\}
\text{.}%
\]

\end{proof}

\begin{corollary}
Let $V\neq\left\{  0\right\}  $ and let $\alpha\geq w\left(  V\right)  $ be an
infinite cardinal number. Let $M$ be a nonempty subset of $V$. If $M$ is
pointwise $\alpha$-lineable if and only if $M$ has $\left(  n,\alpha\right)
$-\textbf{P.L.S.P }for each\textbf{ }$n\in\mathbb{N}$.
\end{corollary}

\begin{corollary}
Let $V\neq\left\{  0\right\}  $ and let $\alpha\geq w\left(  V\right)  $ be an
infinite cardinal number. If $W\neq\left\{  0\right\}  $ is a proper vector
subspace of $V$ then%
\[
V\setminus W\text{ has }\left(  n,\operatorname{codim}W\right)
\text{-\textbf{P.L.S.P }if and only if }\operatorname{codim}W=\infty\text{.}%
\]

\end{corollary}

\section{Infinite pointwise lineability: An extended result\bigskip}

Still on the pointwise perspective, M. Caleder\'{o}n-Moreno, P. Gerlach-Mena
and J. Prado-Bassas present in \cite{MMB} the concepts of infinite pointwise
lineability and infinite dense pointwise lineability as follows:

\begin{definition}
Let $V$ be a vector space, $M$ a nonempty subset of $V$, and $\alpha$ an
infinite cardinal. We say that $M$ is infinitely pointwise-lineable if, for
every $x\in M$, there exists a family $\mathcal{W}=\left\{  W_{n}%
:n\in\mathbb{N}\right\}  $ of vector subspaces of $V$ such that:

\begin{enumerate}
\item[$\left(  i\right)  $] for each $n\in\mathbb{N}$, $\dim\left(
W_{n}\right)  =\alpha$,

\item[$\left(  ii\right)  $] for each $n\in\mathbb{N}$, $x\in W_{n}\subset
M\cup\left\{  0\right\}  $ and

\item[$\left(  iii\right)  $] $W_{m}\cap W_{n}=\mathbb{K}x$ whenever $m$ and
$n$ are distinct positive integers.
\end{enumerate}

\noindent When $V$ is endowed with a topology and $W_{n}$ is a dense subspace
of $V$ for each $n\in\mathbb{N}$, we say that $M$ is infinitely pointwise
$\alpha$-dense lineable in $V$.
\end{definition}

Evidently, these more restrictive notions of lineability presented here
recover ordinary concepts.

The authors verify in \cite{MMB} that, if $\alpha$ is an infinite cardinal,
then the notions of pointwise $\alpha$-lineability and infinite pointwise
$\alpha$-lineability are equivalent. However, the same cannot be said of the
notions of pointwise $\alpha$-dense lineability and infinite pointwise
$\alpha$-dense lineability. In their main result, they state that

\begin{theorem}
\label{Teo1.1}(See \cite[Theorem 2.3]{MMB}) Let $V$ be a metrizable separable
topological vector space, and $\alpha$ be an infinite cardinal number, and $M$
be a nonempty subset of $V$ for which there is a nonempty subset $N$ of $V$
such that

\begin{enumerate}
\item[$\left(  i\right)  $] $M$ is stronger than $N$;

\item[$\left(  ii\right)  $] $M\cap N=\varnothing$;

\item[$\left(  iii\right)  $] $N$ is dense-lineable.
\end{enumerate}

\noindent If $M$ is pointwise $\alpha$-lineable, then $M$ is infinite
pointwise $\alpha$-dense lineable (and therefore pointwise $\alpha$-dense lineable).
\end{theorem}

In this section we extend the above result by removing the assumptions of
separability and metrizability of the vector space $V$.

\section{Main result}

Under the conditions of Theorem \ref{Teo1.1}, $w\left(  V\right)  =\aleph_{0}$
and thus, if $\alpha$ is an infinite cardinal, then $\alpha\geq\aleph
_{0}=w\left(  V\right)  $. Therefore, if we remove the metrizability and
separability assumptions of $V$ and instead require that $V$ be a topological
vector space and that the infinite cardinal $\alpha$ satisfies the condition
$\alpha\geq w\left(  V\right)  $ we will have an extension of Theorem
\ref{Teo1.1}. This is what we do in Theorem \ref{Teo2.1}.

\begin{theorem}
\label{Teo2.1}Let $V$ be a topological vector space and let $\alpha\geq
w\left(  V\right)  $ be an infinite cardinal. Let $M$ be a nonempty subset of
$V$ for which there is a nonempty subset $N$ of $V$ such that

\begin{enumerate}
\item[$\left(  i\right)  $] $M$ is stronger than $N$;

\item[$\left(  ii\right)  $] $M\cap N=\varnothing$;

\item[$\left(  iii\right)  $] $N$ is dense lineable in $V$.
\end{enumerate}

\noindent If $M$ is pointwise $\alpha$-lineable, then $M$ is infinite
pointwise $\alpha$-dense lineable.
\end{theorem}

\begin{proof}
Let $\left\{  U_{i}:i\in I\right\}  $ be a basis for the topology of $V$ with
$\operatorname{card}I=w\left(  V\right)  $. Since $N$ is dense lineable, for
each $i\in I$, we can choose $v_{i}\in\left(  N\cup\left\{  0\right\}
\right)  \cap U_{i}$ such that%
\[
\operatorname*{span}\left\{  v_{i}:i\in I\right\}  \subset N\cup\left\{
0\right\}  \text{.}%
\]
Without loss of generality we can assume that $0=v_{i_{0}}$ for some $i_{0}\in
I$. Since $M$ is pointwise $\alpha$-lineable, given $x\in M$, there is a
subspace $W\subset M\cup\left\{  0\right\}  $ such that%
\[
x\in W\text{ \ \ \ \ and \ \ \ \ }\dim\left(  W\right)  =\alpha\text{.}%
\]
Let $\left\{  w_{j}:j\in J\right\}  $ be a Hamel basis of $W$. If $x\neq0$, we
can assume that $x=w_{j_{0}}$ for some $j_{0}\in J$. Let $\left\{  J_{n}%
:n\in\mathbb{N}\right\}  $ be an enumerable partition of $J$ into subsets of
cardinality $\alpha$. Let us define%
\[
L_{n}=\left\{
\begin{array}
[c]{ll}%
J_{n}\text{,} & \text{if }x=0\text{,}\\
J_{n}\cup\left\{  j_{0}\right\}  \text{,} & \text{if }x\neq0\text{.}%
\end{array}
\right.
\]
and consider $I_{n}\subset L_{n}\setminus\left\{  j_{0}\right\}  $ such that
$\operatorname*{card}\left(  I_{n}\right)  =w\left(  V\right)  $ and let
$\sigma_{n}\colon I_{n}\rightarrow I$ a bijection. For each $i\in I$ let%
\[
-v_{i}+U_{i}=\left\{  -v_{i}+v:v\in U_{i}\right\}  \text{.}%
\]
Hence, for each $i\in I$, $-v_{i}+U_{i}$ is a neighbourhood of the origin and,
therefore, it follows from the continuity of scalar multiplication that, for
each $j\in I_{n}$, there is $\varepsilon_{j}>0$ such that%
\[
\varepsilon_{j}w_{j}\in-v_{\sigma_{n}\left(  j\right)  }+U_{\sigma_{n}\left(
j\right)  }%
\]
that is, such that%
\[
\varepsilon_{j}w_{j}+v_{\sigma_{n}\left(  j\right)  }\in U_{\sigma_{n}\left(
j\right)  }\text{.}%
\]
If $j\in L_{n}\setminus I_{n}$ we consider $\varepsilon_{j}=1$. For each
$n\in\mathbb{N}$, letting $\pi_{n}\colon L_{n}\rightarrow I$ be the function
defined by%
\[
\pi_{n}\left(  j\right)  =\left\{
\begin{array}
[c]{ll}%
\sigma_{n}\left(  j\right)  \text{,} & \text{if }j\in I_{n}\text{,}\\
i_{0}\text{,} & \text{otherwise,}%
\end{array}
\right.
\]
let us define
\[
D_{n}=\operatorname*{span}\left\{  \varepsilon_{j}w_{j}+v_{\pi_{n}\left(
j\right)  }:j\in L_{n}\right\}
\]
It is obvious that $D_{n}$ is a dense subspace of $V$ and $x\in D_{n}$. Let
$\alpha_{1},\ldots,\alpha_{m}$ be scalars not all null and $j_{1},\ldots
,j_{m}\in L_{n}$. Assume that%
\[%
%TCIMACRO{\tsum \limits_{k=1}^{m}}%
%BeginExpansion
{\textstyle\sum\limits_{k=1}^{m}}
%EndExpansion
\alpha_{k}\left(  \varepsilon_{j_{k}}w_{j_{k}}+v_{\pi_{n}\left(  j_{k}\right)
}\right)  =0\text{.}%
\]
In this case, the linear independence of the vectors $w_{j_{k}}$ assures that%
\[
0\neq%
%TCIMACRO{\tsum \limits_{k=1}^{m}}%
%BeginExpansion
{\textstyle\sum\limits_{k=1}^{m}}
%EndExpansion
\alpha_{k}\varepsilon_{j_{k}}w_{j_{k}}=-%
%TCIMACRO{\tsum \limits_{k=1}^{m}}%
%BeginExpansion
{\textstyle\sum\limits_{k=1}^{m}}
%EndExpansion
\alpha_{k}v_{\pi_{n}\left(  j_{k}\right)  }\in M\cap N
\]
and this is a contradiction. Thus,%
\[%
%TCIMACRO{\tsum \limits_{k=1}^{m}}%
%BeginExpansion
{\textstyle\sum\limits_{k=1}^{m}}
%EndExpansion
\alpha_{k}\left(  \varepsilon_{j_{k}}w_{j_{k}}+v_{\pi_{n}\left(  j_{k}\right)
}\right)  \neq0
\]
and $\left\{  \varepsilon_{j}w_{j}+v_{\pi_{n}\left(  j\right)  }:j\in
L_{n}\right\}  $ is linearly independent. It follows that%
\[
\dim\left(  D_{n}\right)  =\operatorname*{card}\left(  L_{n}\right)
=\alpha\text{.}%
\]
If $v\in D_{n}\setminus\left\{  0\right\}  $ then there are non-null scalars
$\lambda_{1},\ldots,\lambda_{m}$ and $j_{1},\ldots,j_{m}\in L_{n}$, such that%
\[
v=%
%TCIMACRO{\tsum \limits_{k=1}^{m}}%
%BeginExpansion
{\textstyle\sum\limits_{k=1}^{m}}
%EndExpansion
\lambda_{k}\left(  \varepsilon_{j_{k}}w_{j_{k}}+v_{\pi_{n}\left(
j_{k}\right)  }\right)  =%
%TCIMACRO{\tsum \limits_{k=1}^{m}}%
%BeginExpansion
{\textstyle\sum\limits_{k=1}^{m}}
%EndExpansion
\lambda_{k}\varepsilon_{j_{k}}w_{j_{k}}+%
%TCIMACRO{\tsum \limits_{k=1}^{m}}%
%BeginExpansion
{\textstyle\sum\limits_{k=1}^{m}}
%EndExpansion
\lambda_{k}v_{\pi_{n}\left(  j_{k}\right)  }\text{.}%
\]
Obviously,
\[%
%TCIMACRO{\tsum \limits_{k=1}^{m}}%
%BeginExpansion
{\textstyle\sum\limits_{k=1}^{m}}
%EndExpansion
\lambda_{k}\varepsilon_{j_{k}}w_{j_{k}}\in M\text{ \ \ \ \ and \ \ \ \ }%
%TCIMACRO{\tsum \limits_{k=1}^{m}}%
%BeginExpansion
{\textstyle\sum\limits_{k=1}^{m}}
%EndExpansion
\lambda_{k}v_{\pi_{n}\left(  j_{k}\right)  }\in N\cup\left\{  0\right\}
\text{.}%
\]
Consequently,%
\[
v\in M+\left(  N\cup\left\{  0\right\}  \right)  \subset M\text{,}%
\]
and this proves that $D_{n}\subset M\cup\left\{  0\right\}  $. It only remains
to prove that, $D_{m}\cap D_{n}=\mathbb{K}x$ whenever $m$ and $n$ are
different positive integers. Let $m,n\in\mathbb{N}$, with $m\neq n$, and let
$v\in D_{m}\cap D_{n}$. Therefore, there are $\mu_{1},\ldots,\mu_{r+s}%
,\alpha,\beta\in\mathbb{K}$, $j_{1},\ldots,j_{r}\in L_{m}$, and $j_{r+1}%
,\ldots,j_{r+s}\in L_{n}$ such that%
\[%
%TCIMACRO{\tsum \limits_{k=1}^{r}}%
%BeginExpansion
{\textstyle\sum\limits_{k=1}^{r}}
%EndExpansion
\mu_{k}\left(  \varepsilon_{j_{k}}w_{j_{k}}+v_{\pi_{m}\left(  j_{k}\right)
}\right)  +\alpha x=v=%
%TCIMACRO{\tsum \limits_{k=r+1}^{r+s}}%
%BeginExpansion
{\textstyle\sum\limits_{k=r+1}^{r+s}}
%EndExpansion
\mu_{k}\left(  \varepsilon_{j_{k}}w_{j_{k}}+v_{\pi_{n}\left(  j_{k}\right)
}\right)  +\beta x
\]
Hence, making%
\[
\eta_{k}=\left\{
\begin{array}
[c]{ll}%
\mu_{k}\text{,} & \text{if }1\leq k\leq r\text{,}\\
-\mu_{k}\text{,} & \text{if }r+1\leq k\leq r+s\text{,}%
\end{array}
\right.
\]
and defining $\pi\colon J\rightarrow I$ by $\pi\left(  j\right)  =\pi
_{n}\left(  j\right)  $ if $j\in L_{n}$, we have%
\[%
%TCIMACRO{\tsum \limits_{k=1}^{r+s}}%
%BeginExpansion
{\textstyle\sum\limits_{k=1}^{r+s}}
%EndExpansion
\eta_{k}\varepsilon_{j_{k}}w_{j_{k}}+\left(  \alpha-\beta\right)  x=-%
%TCIMACRO{\tsum \limits_{k=1}^{r+s}}%
%BeginExpansion
{\textstyle\sum\limits_{k=1}^{r+s}}
%EndExpansion
\eta_{k}v_{\pi\left(  j_{k}\right)  }\text{.}%
\]
If it were%
\[%
%TCIMACRO{\tsum \limits_{k=1}^{r+s}}%
%BeginExpansion
{\textstyle\sum\limits_{k=1}^{r+s}}
%EndExpansion
\eta_{k}\varepsilon_{j_{k}}w_{j_{k}}+\left(  \alpha-\beta\right)  x=-%
%TCIMACRO{\tsum \limits_{k=1}^{r+s}}%
%BeginExpansion
{\textstyle\sum\limits_{k=1}^{r+s}}
%EndExpansion
\eta_{k}v_{\pi\left(  j_{k}\right)  }\neq0\text{,}%
\]
then we would have%
\[
-%
%TCIMACRO{\tsum \limits_{k=1}^{r+s}}%
%BeginExpansion
{\textstyle\sum\limits_{k=1}^{r+s}}
%EndExpansion
\eta_{k}v_{\pi\left(  j_{k}\right)  }\in M\cap N
\]
which would be a contradiction. Thus,%
\[%
%TCIMACRO{\tsum \limits_{k=1}^{r+s}}%
%BeginExpansion
{\textstyle\sum\limits_{k=1}^{r+s}}
%EndExpansion
\eta_{k}\varepsilon_{j_{k}}w_{j_{k}}+\left(  \alpha-\beta\right)  x=0
\]
and since the vectors involved are linearly independent, we have%
\[
\eta_{1}=\cdots=\eta_{r+s}=\alpha-\beta=0\text{.}%
\]
This shows that $v=\alpha x\in\mathbb{K}x$ and the proof is done.
\end{proof}

\begin{corollary}
Let $V\neq\left\{  0\right\}  $ and $W$ be a non-trivial dense linear subspace
of $V$ with $\operatorname{codim}W=\infty$. If $w(V)\leq\operatorname{codim}%
W$\ then $V\setminus W$ is infinite pointwise $\operatorname{codim}W$-dense lineable.
\end{corollary}

\begin{proof}
Considering $M=V\setminus W$ and $N=W$, we have $M+N\subset M$ and $M\cap
N=\varnothing$ and the result follows by Theorem \ref{Teo2.1}.
\end{proof}

\begin{corollary}
Let $V\neq\left\{  0\right\}  $ and let $\alpha>w\left(  V\right)  $ be an
infinite cardinal. Let $M$ be a nonempty subset of $V$ for which there is a
nonempty subset $N$ of $V$ such that

\begin{enumerate}
\item[$\left(  i\right)  $] $M$ is stronger than $N$;

\item[$\left(  ii\right)  $] $M\cap N=\varnothing$;

\item[$\left(  iii\right)  $] $N$ is dense lineable in $V$.
\end{enumerate}

\noindent If $M$ is pointwise $\alpha$-lineable, then $M$ has $\left(
2,w\left(  V\right)  \right)  $-\textbf{D}.\textbf{P.L.S.P} for each
$n\in\mathbb{N}$.
\end{corollary}

\begin{proof}
Assume that $M$ is pointwise $\alpha$-lineable. By Theorem \ref{Teo2.1} we can
infer that $M$ is pointwise $\alpha$-dense lineable. Hence, we can invoke
Corollary \ref{First corollary in DPSP} to obtain that $M$ has $\left(
2,w\left(  V\right)  \right)  $-\textbf{D}.\textbf{P.L.S.P}.
\end{proof}

\end{document}